\documentclass[
final
]{dmtcs-episciences}

\usepackage[utf8]{inputenc}
\usepackage{subfigure}

\usepackage[round]{natbib}

\usepackage{amsmath}
\usepackage{amssymb}

\newcommand{\ex}{\mathrm{ex}}
\newcommand{\eps}{\varepsilon}

\newtheorem{theorem}{Theorem}[section]
\newtheorem{lemma}[theorem]{Lemma}
\newtheorem{proposition}[theorem]{Proposition}
\newtheorem{claim}[theorem]{Claim}
\newtheorem{corollary}[theorem]{Corollary}

\author{Andrzej Dudek\affiliationmark{1}\thanks{Supported in part by Simons Foundation Grant \#522400.}
  \and Andrzej Ruci\'nski\affiliationmark{2}\thanks{Supported in part by the Polish NSC grant 2014/15/B/ST1/01688}}
\title[Monochromatic loose paths in multicolored $k$-uniform cliques]{Monochromatic loose paths in multicolored $k$-uniform cliques}
\affiliation{
  Western Michigan University, USA\\
  Adam Mickiewicz University, Poland}
\received{2018-3-15}

\accepted{2019-7-16}

\begin{document}
\publicationdetails{21}{2019}{4}{7}{4372}
\maketitle
\begin{abstract}
  For integers $k\ge 2$ and $\ell\ge 0$,  a $k$-uniform hypergraph is called  a \emph{loose path of length~$\ell$}, and denoted by  $P_\ell^{(k)}$, if it consists of $\ell $ edges $e_1,\dots,e_\ell$ such that $|e_i\cap e_j|=1$ if $|i-j|=1$ and  $e_i\cap e_j=\emptyset$ if $|i-j|\ge2$. In other words,  each pair of consecutive edges intersects on a single vertex, while all other pairs are disjoint. Let $R(P_\ell^{(k)};r)$ be the minimum integer $n$ such that every $r$-edge-coloring of the complete $k$-uniform hypergraph $K_n^{(k)}$ yields a monochromatic copy of $P_\ell^{(k)}$.
In this paper we are mostly interested in \emph{constructive} upper bounds on $R(P_\ell^{(k)};r)$, meaning that on the cost of possibly enlarging the order of the complete hypergraph, we would like to efficiently find a monochromatic copy of $P_\ell^{(k)}$ in every coloring. In particular, we show that there is a constant $c>0$ such that for all $k\ge 2$, $\ell\ge3$, $2\le r\le k-1$, and $n\ge k(\ell+1)r(1+\ln(r))$, there is an algorithm such that for every $r$-edge-coloring of the edges of $K_n^{(k)}$, it finds a monochromatic copy of $P_\ell^{(k)}$ in time at most $cn^k$. We also prove a  non-constructive upper bound $R(P_\ell^{(k)};r)\le(k-1)\ell r$.
\end{abstract}

\section{Introduction}

For positive integers $k\ge2$ and $\ell\ge0$,  a $k$-uniform hypergraph is called  a \emph{loose path of length~$\ell$}, and denoted by  $P_\ell^{(k)}$, if its vertex set is $\{v_1, v_2, \ldots, v_{(k-1)\ell+1}\}$ and the edge set is $\{e_i = \{ v_{(i-1)(k-1)+q} : 1 \le q \le k \},\  i=1,\dots,\ell\}$, that is, for $\ell\ge2$, each pair of consecutive edges intersects on a single vertex (see Figure~\ref{fig:loose}), while for $\ell=0$ and $\ell=1$ it is, respectively, a single vertex and an edge. For $k=2$ the loose path $P_{\ell}^{(2)}$ is just a (graph) path on $\ell+1$ vertices.

Let ${H}$ be a $k$-uniform hypergraph and $r\ge2$ be an integer. The \emph{multicolor Ramsey number} $R(H;r)$ is the minimum $n$ such that every $r$-edge-coloring of the complete $k$-uniform hypergraph $K_n^{(k)}$ yields a monochromatic copy of ${H}$.

\begin{figure}[t]
\centering
\includegraphics[scale = 0.7]{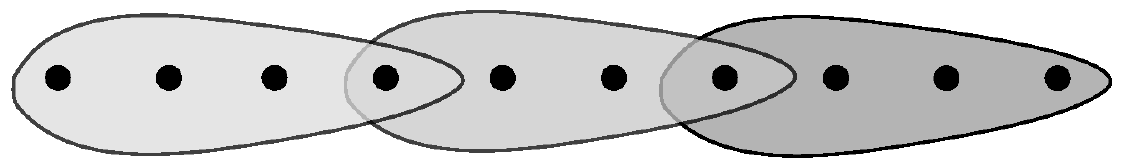}
\caption{A 4-uniform loose path $P_3^{(4)}$.}
\label{fig:loose}
\end{figure}

\subsection{Known results for graphs}
For graphs, determining the Ramsey number $R(P_\ell^{(2)},r)$ is a well-known problem that attracted a lot of attention.
It was shown by \cite{GG67} that
\begin{equation*}
R(P_\ell^{(2)},2) = \left\lfloor \frac{3\ell+1}{2} \right\rfloor.
\end{equation*}
For three colors  \cite{FL07} proved that $R(P_\ell^{(2)},3)\approx 2\ell$. Soon after,
\cite{GRSS07,GRSS07b} determined this number exactly, showing that for all sufficiently large $\ell$
\begin{equation}\label{thm:grss}
R(P_\ell^{(2)},3) =
\begin{cases}
2\ell+1 & \text{ for even } \ell,\\
2\ell & \text{ for odd } \ell,
\end{cases}
\end{equation}
as conjectured earlier by \cite{FS75}.  For $r\ge 4$ much less is known. A celebrated Tur\'an-type  result of \cite{EG59} implies that
\begin{equation}\label{EG}
R(P^{(2)}_\ell,r) \le r\ell.
 \end{equation}
 Recently, this was slightly improved by \cite{S15} and, subsequently, by \cite{DJR2017} who showed that for all sufficiently large $\ell$,
\begin{equation}\label{thm:djr}
R(P_\ell^{(2)};r) \le (r-1/4)(\ell+1).
\end{equation}

\subsection{Known results for hypergraphs}
 Let us first recall what is known about $R(P_\ell^{(k)},r)$ for $k\ge 3$. For two colors, \cite{GR2012} considered only paths of length $\ell=2,3,4$ and proved that $R(P_2^{(k)},2)=2k-1$, $R(P_3^{(k)},2)=3k-1$, and $R(P_4^{(k)},2)=4k-2$.
Later, for $k=3$ or $k\ge8$, and  $\ell\ge 3$, \cite{OS2014, OS2017} determined this number completely:
\[
R(P_\ell^{(k)},2) = (k-1)\ell + \left\lfloor \frac{\ell+1}{2}  \right\rfloor,
\]
and conjectured that the above formula is also valid for $k=4,5,6,7$.

For an arbitrary number of colors there are only few results and mainly for very short paths. The following is known.

 For $\ell=2$ and $k=3$ (so called \emph{bows}), \cite{AGLM2014} determined the value of $R(P_2^{(3)},r)$ for an infinite subsequence of integers $r$ (including $2\le r\le 10$) and for $r\to\infty$ they showed that $R(P_2^{(3)},r)\approx \sqrt{6r}$.

For $k\ge4$, and large $r$, the Ramsey number $R(P_2^{(k)},r)$ can be easily upper bounded by a standard application of Tur\'an  numbers (by  counting the average number of edges per color). Recall that for a given $k$-uniform hypergraph~$H$, the \emph{Tur\'an number}, $\ex_k(n; H)$, is the maximum number of edges in an $n$ vertex $k$-uniform hypergraph with no copy of~$H$.
It was proved by \cite{F1977} that $\ex_k(n;P_2^{(k)})=\binom{n-2}{k-2}$ for $n$ sufficiently large, from which it follows quickly that $R(P_2^{(k)},r)\le\sqrt{k(k-1)r}$.

For $\ell\ge3$, a similar approach via Tur\'an numbers $\ex_k(n;P_\ell^{(k)})$, determined for large $n$ by \cite{FJS2014}, yields for  large $r$,
 \begin{equation}\label{TRL}
R(P_\ell^{(k)};r) \le k\ell r/2,
\end{equation}
and, slightly better for $\ell=3$,
 \begin{equation}\label{TR3}
R(P_3^{(k)};r) \le k r.
\end{equation}

In the smallest instance $k=\ell=3$, owing to the validity of formula $\ex_3(n;P_3^{(3)})=\binom{n-1}2$ for all $n\ge8$ (see \cite{JPR2016}), the above bound holds for all $r\ge3$:
\begin{equation}\label{TR}
R(P_3^{(3)},r)\le 3r.
\end{equation}
Recently,  \L{}uczak and Polcyn twice improved (\ref{TR}) significantly. First, in \cite{LP2017}, they showed that $R(P_3^{(3)},r)\le2r+O(\sqrt r)$, then, in \cite{LP2018}, they broke the barrier of $2r$ by proving the bound $R(P_3^{(3)},r)<1.98r$, both results for large $r$.
This still seems to be far from the true value which is conjectured to be equal to $r+6$, the current best lower bound.
In a series of papers Jackowska, Polcyn, and Ruci\'nski (\cite{JPR2016}, \cite{P2017}, \cite{PR2017}) confirmed this conjecture for  $r\le10$. Finally, for
$\ell=3$, $k$ arbitrary, and $r$ large, \cite{LPR2017} showed an upper bound $R(P_3^{(k)},r)\le250r$ which is independent of $k$.

 In the next section we  show a general upper bound, obtained iteratively for all $k\ge2$, starting from the Erd\H os-Gallai bound (\ref{EG}) $R(P^{(2)}_\ell,r) \le r\ell$.
\begin{theorem}\label{non}
For all $k\ge 2$, $\ell\ge3$, and $r\ge2$ we have $R(P_\ell^{(k)};r) \le (k-1)\ell r$.
\end{theorem}

Theorem~\ref{non} can be easily improved for $r\ge 3$ provided $\ell$ is large. Using~\eqref{thm:grss} instead of (\ref{EG}), we obtain for three colors that
\[
R(P_\ell^{(k)};3) \le (3k-4)\ell,
\]
and for $r\ge4$, by~\eqref{thm:djr},
\[
R(P_\ell^{(k)};r) \le (k-1)\ell r - \ell/4.
\]
On the other hand, for large~$r$, the bound (\ref{TRL}) is roughly
 twice better than the one in Theorem~\ref{non}.

\subsection{Constructive bounds}

In this paper we are mostly interested in \emph{constructive} bounds which means that on the cost of possibly enlarging the order of the complete hypergraph, we would like to efficiently find a monochromatic copy of a target hypergraph $F$ in every coloring. Clearly, by examining all copies of $F$ in $K^{(k)}_n$ for $n\ge R(F;r)$, we can always find a monochromatic one in time $O(n^{|V(F)|})$. Hence, we are interested in complexity not depending on $F$, preferably $O(n^k)$.
Given a $k$-graph $F$, a constant $c>0$ and integers $r$ and $n$, we say that a property $\mathcal{R}(F,r,c,n)$ holds if there is an algorithm such that for every $r$-edge-coloring of the edges of $K_n^{(k)}$, it finds a monochromatic copy of $F$ in time at most $cn^k$. For graphs, a constructive result of this type can be deduced from the proof of Lemma 3.5 in \cite{DP2017}.

\begin{theorem}[\cite{DP2017}]\label{con3}
There is a constant $c>0$ such that for all $\ell\ge3$, $r\ge 2$, and $n\ge 2^{r+1} \ell$, property $\mathcal{R}(P_\ell^{(2)},r,c,n)$ holds.
\end{theorem}

Our goal is to obtain  similar constructive results for loose hyperpaths.
In Section 2, we show that, by replacing the Erd\H os-Gallai bound (\ref{EG}) with the assumption on $n$ given in Theorem~\ref{con3}, the proof of Theorem~\ref{non} can be easily adapted to yield a constructive result.
\begin{theorem}\label{con2}
There is a constant $c>0$ such that for all $k\ge 2$, $\ell\ge3$, $r\ge 2$, and $n\ge 2^{r+1}\ell + (k-2)\ell r$, property $\mathcal{R}(P_\ell^{(k)},r,c,n)$ holds.
\end{theorem}

Our second constructive bound (valid only for $r\le k$) utilizes a  more sophisticated algorithm.
\begin{theorem}\label{con}
There is a constant $c>0$ such that for all $k\ge 2$, $\ell\ge3$, $2\le r\le k$, and $n\ge k(\ell+1)r\left(1+\frac{1}{k-r+1}+\ln\left(1+\frac{r-2}{k-r+1}\right)\right)$, property $\mathcal{R}(P_\ell^{(k)},r,c,n)$ holds. For $r=2$, the bound on $n$ can be improved to $n\ge (2k-2)\ell+k$.
\end{theorem}
\noindent
Note that for $r=2$ the lower bound on $n$ in Theorem \ref{con} is very close to that in~Theorem~\ref{non}.
For $r=k$ the bound in Theorem \ref{con} assumes a simple form
$$n\ge k^2(\ell+1)(2+\ln(k-1).$$ Furthermore, when $r\le k-1$, one can show (see Claim~\ref{claim:cor}) that  $$\frac{1}{k-r+1}+\ln\left(1+\frac{r-2}{k-r+1}\right) \le \ln\left(1+\frac{r-1}{k-r}\right)$$ yielding the following corollary.

\begin{corollary}\label{cor:1}
There is a constant $c>0$ such that for all $k\ge 3$, $\ell\ge3$, $2\le r\le k-1$, and $n\ge k(\ell+1)r\left(1+\ln\left(1+\frac{r-1}{k-r}\right)\right)$, property $\mathcal{R}(P_\ell^{(k)},r,c,n)$ holds.
\end{corollary}
\noindent
We can further replace the lower bound on $n$ in Corollary~\ref{cor:1} by (slightly weaker but simpler)
$$n\ge k(\ell+1)r(1+\ln r).$$

Observe that in several instances the lower bound in Theorem~\ref{con} (and also in Corollary~\ref{cor:1}) is significantly better (that means smaller) than the one in Theorem~\ref{con2} (for example for large $k$ and  $k/2 \le r \le  k$). On the other hand, for some instances bounds in Theorems~\ref{con2} and \ref{con} are basically the same. For example, for fixed $r$, large~$k$ and $\ell\ge k$ the lower bound is $k\ell r + o(k\ell)$. This also matches the bound from Theorem~\ref{non}.

\section{Proof of Theorems \ref{non} and \ref{con2}}

For completeness, we begin with proving bounds (\ref{TRL})-(\ref{TR}). 

\begin{proposition} For all $k\ge3$ and $\ell\ge3$, inequalities (\ref{TRL}) and (\ref{TR3}) hold for large $r$, while inequality (\ref{TR3}) holds for all $r$.
\end{proposition}

\begin{proof} It has been proved in \cite{FJS2014} and \cite{KMV2017} that for all $k\ge3$ and $\ell\ge3$, except for $k=\ell=3$ (but see Acknowledgements in \cite{JPR2016}), and for $n$ sufficiently large
$$\ex_k(n; P_\ell^{(k)})=\binom nk-\binom{n-t}k+\delta_\ell\binom{n-k-t}{k-2},$$
where $\delta_\ell=0$ if $\ell$ is odd and, otherwise, $\delta_\ell=1$, while $t=\lfloor\tfrac{\ell+1}2\rfloor-1$. Regardless of the parity of $\ell$, for every $\eps>0$ and sufficiently large $n$, this Tur\'an number is smaller than $ (1+\eps)tn^{k-1}/(k-1)!$. With some foresight, we require that $\eps\le(2\ell-1)^{-1}$.
Thus,  for fixed $\eps>0$, $k\ge 2$ and $\ell \ge 3$ and all sufficiently large $r$,  the average number of edges per color in an $r$-coloring of the complete $k$-graph $K_n^{(k)}$ with $n\ge \ell kr/2$ is
$$\frac{\binom nk}r\ge(1-\eps)\frac{n^k}{rk!}\ge(1-\eps)\frac{\ell n^{k-1}}{2(k-1)!}\ge (1+\eps)\frac{(\ell-1)n^{k-1}}{2(k-1)!}>\ex_k(n; P_\ell^{(k)}),$$
 which proves (\ref{TRL}).

For $\ell=3$, the formula for $\ex_k(n; P_\ell^{(k)})$ simplifies to $\ex_k(n; P_3^{(k)})=\binom{n-1}{k-1}$ and we have
$$\frac{\binom nk}r\ge\binom{n-1}{k-1}$$ already for $n\ge kr$. Since the only extremal $k$-graph in this case is the full star and it is impossible that all colors are stars, we get  (\ref{TR3}).

Finally, for $\ell=k=3$ it was proved in \cite{JPR2016} that $\ex_3(n; P_3^{(3)})=\binom{n-1}{2}$ for all $n\ge8$ and the same argument as above applies to all $r\ge3$. 
\end{proof}

Preparing for the proof of Theorem \ref{non}, recall that
\cite{EG59} showed that the Tur\'an number for a graph path $P_\ell^{(2)}$ satisfies the bound $\ex_2(n;P^{(2)}_\ell)\le\tfrac12(\ell-1)n$. This immediately yields, by the same  argument as in the above proof,
that the majority color in $K_{r\ell}$ contains a copy of $P_\ell^{(2)}$, and consequently $R(P_\ell^{(2)};r)\le r\ell$.
We are going to use this result by blowing up the edges of a graph to obtain a 3-graph, then blowing the edges of a 3-graph to obtain a 4-graph, and so on. Formally, we call an edge of a hypergraph \emph{selfish} if it contains a vertex of degree one, that is, a vertex which belongs exclusively to this edge. We call a hypergraph $H$ \emph{selfish} if every edge of $H$ is selfish. Clearly, for $k\ge3$ and $\ell\ge 1$, the loose path $P^{(k)}_\ell$ is selfish.

A selfish $k$-graph $H$ can be reduced to a $(k-1)$-graph $G_H$ by removing one vertex of degree one from each edge of $H$.
Inversely, every $(k-1)$-graph $G$ can be turned into a selfish $k$-graph $H$, called \emph{a selfish extension of $G$}, such that $G=G_H$, by adding $|E(G)|$ vertices, one to each edge of $G$.
Note that $|E(H)|=|E(G_H)|$.

\begin{figure}
\centering
\includegraphics[scale = 0.7]{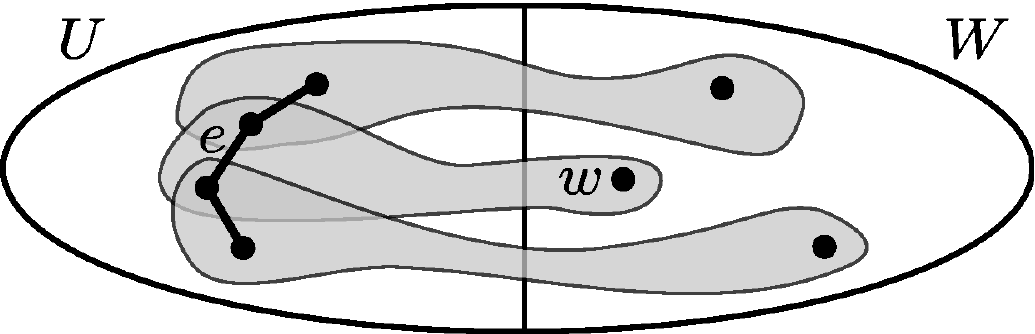}
\caption{Proving Lemma~\ref{self}.}
\label{fig:self}
\end{figure}

\begin{lemma}\label{self}
For a given integer $k\ge 3$, let $H$ be a selfish $k$-graph with $G=G_{H}$. Then
\[
R(H;r)\le R(G;r)+r(|E(H)|-1)+1.
\]
\end{lemma}

\begin{proof} Let $n=R(G;r)+r(|E(H)|-1)+1$, $V=U\cup W$, $U\cap W=\emptyset$,  $|V|=n$, $|U|=R(G;r)$, and $|W|=r(|E(H)|-1)+1$.  Consider an $r$-coloring of the edges of $K_n^{(k)}$. For every $(k-1)$-tuple $e$ of vertices in $U$, we choose the most frequent color on all the $k$-tuples $e\cup\{w\}$, $w\in W$ (see Figure~\ref{fig:self}). This induces an $r$-coloring of the edges of the clique $K_{|U|}^{(k-1)}$ on vertex set $U$. By the definition of $R(G;r)$, this yields  a monochromatic (say, \emph{red}) copy $G'$ of $G$ (in the induced coloring).  Note that for each edge $e$ of $G'$ the \emph{red} color appears on at least $|E(H)|$ $k$-tuples containing $e$ and, thus, one can find a \emph{red} selfish extension $H'$ of~$G'$ which is isomorphic to $H$.
\end{proof}

\begin{proof}[of Theorem \ref{non}.]
We use induction on $k$. For $k=2$ the theorem coincides with the Erd\H os-Gallai result (\ref{EG}). Assume that for some $k\ge3$ we have $R(P_\ell^{(k-1)};r) \le (k-2)\ell r$, observe that  $P_\ell^{(k-1)}=G_{P_\ell^{(k)}}$, and apply Lemma \ref{self} obtaining
\[
R(P_\ell^{(k)};r) \le R(P_\ell^{(k-1)};r)+r(\ell-1)+1
\le (k-2)\ell r + \ell r = (k-1)\ell r.
\]
\end{proof}

\begin{proof}[of Theorem~\ref{con2}]
To get the desired lower bound on $n$ it suffices  to replace in the base step of induction the Erd\H{o}s-Gallai bound (\ref{EG}) by the one from Theorem~\ref{con3} which yields
\[
R(P_\ell^{(k)};r) \le R(P_\ell^{(k-1)};r)+r(\ell-1)+1
\le 2^{r+1}\ell+(k-3)\ell r + \ell r =2^{r+1}\ell+ (k-2)\ell r.
\]

It remains to show that the performance time does not exceed $cn^k$ for some $c>0$, which, by Theorem \ref{con3}, is the case when $k=2$. Without loss of generality, assume that $c\ge1$. Suppose that for some $k\ge3$ it holds for $(k-1)$-uniform hypergraphs. Similarly as in Lemma~\ref{self}, we arbitrarily partition $V=U \cup W$ with $|U| \ge2^{r+1}\ell + (k-3)\ell r$ and $|W| = r(\ell-1)+1$. Next we color each $(k-1)$-tuple $e$ in $U$ by the most frequent color on the $k$-tuples $e\cup\{w\}$, $w\in W$. This requires no more than
$$\binom {|U|}{k-1}\times |W|\le n^k/(k-1)!$$ steps. Finally, by inductive assumption, in time at most $cn^{k-1}$ we find a monochromatic copy of $P_{\ell}^{(k-1)}$ in $U$ which can be extended to a monochromatic $P_{\ell}^{(k)}$ in no more than
$$
\ell|W| \le r\ell^2\le 2^{r-1}\ell^2\le 2^{r-1} (2^{-r-1}n)^2  = 2^{-r-3}n^2
$$
steps. Altogether, the performance time, using bounds $r\ge2$, $k\ge3$, $\ell\ge3$, and so $n\ge30$, is
$$n^{k}/(k-1)! + cn^{k-1}+2^{-r-3}n^2\le(1/2+c/30 +1/960)n^k\le c n^k,$$ as required.
\end{proof}

\section{Proof of Theorem \ref{con}}

The proof is based on the depth first search (DFS) algorithm. Such approach for graphs was first successfully applied by \cite{BKS2012b, BKS2012} and for Ramsey-type problems by \cite{DP15,DP2017}.

Given integers $k$ and $2\le m\le k$, and disjoint sets of vertices $W_1,\dots, W_{m-1}$, $V_m$, \emph{an $m$-partite complete $k$-graph} $K^{(k)}(W_1,\dots, W_{m-1},V_m)$ consists of all $k$-tuples of vertices with exactly one element in each $W_i$, $i=1,\dots,m-1$, and $k-m+1$ elements in $V_m$. Note that if $|W_i|\ge\ell$, $i=1,\dots,m-1$, and $|V_m|\ge \ell(k-m)+1$ for $m\le k-1$ (or $|V_m|\ge\ell$ for $m=k$), then $K^{(k)}(W_1,\dots, W_{m-1},V_m)$ contains~$P_\ell^{(k)}$. Indeed, if $m\le k-1$, then we inductively find a copy of $P_\ell^{(k)}$ in $K^{(k)}(W_1,\dots, W_{m-1},V_m)$, edge by edge, by making sure that for each edge $e$, $|e\cap W_i| = 1$ (for $i=1,\dots,m-1$) and $|e\cap V_m| = m-k+1$ and the consecutive edges of~$P_\ell^{(k)}$ intersect in $V_m$. In the remaining case, when $m=k$, the consecutive edges of~$P_\ell^{(k)}$ intersect either in $W_1$ or $V_k$ by alternating between these two sets.

We now give a description of the algorithm.
As an input there is an $r$-coloring of the edges of the complete $k$-graph $K_n^{(k)}$.
The algorithm consists of $r-1$ implementations of DFS subroutine, each round exploring  the edges of one color only and either finding a monochromatic copy of $P_\ell^{(k)}$ or decreasing the number of colors present on a large subset of vertices, until   after the $(r-1)$st round we end up with a monochromatic complete $r$-partite subgraph, large enough to contain a copy of~$P_\ell^{(k)}$.

During the $i$th round, while trying to build a copy of the path $P_\ell^{(k)}$ in the $i$th color,  the algorithm selects a subset $W_{i,i}$ from a set of still available vertices $V_i\subseteq V$  and, by the end of the round, creates trash bins $S_i$ and $T_i$.
The search for $P_\ell^{(k)}$ is realized by a DFS process which maintains a working path $P$ (in the form of a sequence of vertices) whose endpoints (the first or the last $k-1$ vertices on the sequence) are either extended to a longer path or otherwise put into $W_{i,i}$. The round  is terminated  whenever $P$ becomes a copy of $P_\ell^{(k)}$ or the size of $W_{i,i}$ reaches certain threshold, whatever comes first. In the latter case we set $S_i=V(P)$.

To better depict the extension process, we introduce the following terminology. An edge of $P_\ell^{(k)}$ is called \emph{pendant} if it contains at most one vertex of degree two.
The vertices of degree one, belonging to the pendant edges of $P_\ell^{(k)}$ are called \emph{pendant}. In particular, in $P_1^{(k)}$ all its $k$ vertices are pendant. For convenience, the unique vertex of the path $P_0^{(k)}$ is also considered to be pendant. Observe that for $t\ge0$, to extend a copy $P$ of $P_t^{(k)}$ to a copy of $P_{t+1}^{(k)}$ one needs to add a new edge which shares exactly one vertex with $P$ and that vertex has to be pendant in $P$. Our algorithm may also come across a situation when $P=\emptyset$, that is, $P$ has no vertices at all. Then by an extension of $P$ we mean any edge whatsoever.

The sets $W_{i,i}$ have a double subscript, because they are updated in the later rounds to $W_{i,i+1}$, $W_{i,i+2}$, and so on, until at the end of the $(r-1)$st round (unless a monochromatic $P_\ell^{(k)}$ has been found) one obtains sets $W_i:= W_{i,r-1}$, $i=1,\dots, r-1$, a final trash set $T = \bigcup_{i=1}^{r-1} T_{i} \cup \bigcup_{i=1}^{r-1} S_{i}$ and the remainder set $V_r=V\setminus(\bigcup_{i=1}^{r-1} W_i\cup T)$ such that all $k$-tuples of vertices in $K^{(k)}(W_1,\dots,W_{r-1},V_r)$ are of color $r$.
As an input of the $i$th round  we take  sets $W_{j,i-1}$, $j=1,\dots, i-1$, and $V_{i-1}$, inherited from the previous round, and rename them to
 $W_{j,i}$, $j=1,\dots, i-1$, and $V_{i}$. We also set $T_i=\emptyset$ and $P=\emptyset$,  and update all these sets dynamically until the round ends.

Now come the details.
For $1\le i\le r-1$, let
\begin{equation}\label{tau}
\tau_i =
\begin{cases}
(i-1)\left(  \frac{\ell}{k-r+1} + \frac{\ell+1}{k-r+2} + \dots + \frac{\ell+1}{k-i} \right) & \text{ if } 1\le i \le r-2,\\[4pt]
(r-2)\frac{\ell}{k-r+1} & \text{ if } i=r-1,
\end{cases}
\end{equation}
and
\[
t_i = \tau_i + 2(i-1).
\]

Note that $\tau_i$ is generally not an integer. It can be easily shown (see Claim~\ref{claim:tau}) that for all $2\le r\le k$ and $1\le i\le r-1$
\begin{equation}\label{claim5}
\tau_i \le (i-1)(\ell+1) \left(\frac{1}{k-r+1}+\ln\left(1+\frac{r-2}{k-r+1}\right)\right).
\end{equation}

Before giving a general description of the $i$th round, we deal separately with the 1st and 2nd round.
\bigskip

\noindent
\textbf{Round 1}

\smallskip

\noindent
Set $V_1=V$, $W_{1,1}=\emptyset$, and $P=\emptyset$. Select an arbitrary  edge $e$ of color one (say, \emph{red}), add its vertices to $P$ (in any order), reset $V_1:=V_1\setminus e$, and try to extend $P$ to a red copy of $P_2^{(k)}$. If successful, we appropriately enlarge $P$, diminish $V_1$, and  try to further extend $P$ to a red copy of $P_3^{(k)}$. This procedure is repeated until finally  we either  find a red copy of $P_\ell^{(k)}$ or, otherwise, end up with a red copy $P$ of $P_t^{(k)}$, for some $1\le t\le\ell-1$, which cannot be extended any more. In the latter case we  shorten $P$ by moving all its pendant vertices to  $W_{1,1}$ and try to extend the remaining red path again. When $t\ge2$, the new path has $t-2$ edges. If $t=2$, $P$ becomes a single vertex path $P_0^{(k)}$, while if $t=1$, it becomes empty.

Let us first  consider the simplest but instructive case $r=2$ in which only one round is performed.
We terminate Round~1 as soon as
\begin{equation}\label{step1:W11:lb1}
 |W_{1,1}|\ge\ell .
\end{equation}
If at some point $P=\emptyset$  and cannot be extended (which means there are no red edges within $V_1$), but (\ref{step1:W11:lb1}) fails to hold, then we move  $\ell-|W_{1,1}|$ arbitrary  vertices from $V_1=V\setminus W_{1,1}$ to $W_{1,1}$ and stop.
At that moment, no edge of $K^{(k)}(W_{1,1},V_1)$ is red (so, all of them must be, say, blue). Moreover, since the size of $W_{1,1}$ increases by increments of at most $2(k-1)$, we have
$$\ell\le |W_{1,1}|  \le\ell + 2(k-1)-1,$$
and, consequently,
$$|V_1|=n-|W_{1,1}|-|V(P)|\ge n-\ell - 2(k-1)+1-|V(P_{\ell-1}^{(k)})|\ge \ell(k-2)+1$$
by our bound on $n$ (see Theorem~\ref{con}, case $r=2$). This means that the completely blue copy of $K^{(k)}(W_{1,1},V_1)$ is large enough to contain a copy of $P_\ell^{(k)}$.

When $r\ge3$, there are still more rounds ahead during which the set $W_{1,1}$ will be cut down, so we need to ensure  it is large enough to survive the entire process.
We terminate Round 1 as soon as
\begin{equation}\label{step1:W11:lb}
 |W_{1,1}|\ge(k-1) \tau_2 +\ell +1 .
\end{equation}
If at some point $P=\emptyset$  and cannot be extended and~\eqref{step1:W11:lb} fails to hold, we move
 $\lceil(k-1)\tau_2\rceil+\ell+1-|W_{1,1}|$ arbitrary  vertices from $V_1=V\setminus W_{1,1}$ to $W_{1,1}$ and stop.

Since the size of $W_{1,1}$ increases by increments of at most $2(k-1)$ and the right-hand side of (\ref{step1:W11:lb}) is not necessarily integer, we also have
\begin{equation}\label{step1:W11:ub}
|W_{1,1}| \le (k-1) \tau_2 +\ell +1+ 2(k-1).
\end{equation}
 Finally, we set $S_1:=P$, $T_1 = \emptyset$ for mere convenience, and  $V_1:=V\setminus (W_{1,1}\cup S_1 \cup T_1)$. Note that $|S_1|\le|V(P_{\ell-1}^{(k)})|=(\ell-1)(k-1)+1$. Also, it is important to realize that no edge of $K^{(k)}(W_{1,1}, V_1)$
is colored red.

\bigskip

\noindent
\textbf{Round 2}

\smallskip

\noindent
 We begin with resetting $W_{1,2}:=W_{1,1}$ and $V_2 := V_1$, and setting $P:=\emptyset$, $W_{2,2}=\emptyset$, and  $T_2 := \emptyset$. In this round only the edges of color  two (say, blue) belonging to  $K^{(k)}(W_{1,2}, V_2)$ are considered. Let us denote the set of these  edges by $E_2$.  We choose an arbitrary edge  $e\in E_2$, set $P=e$, and try to extend $P$ to a copy of $P_2^{(k)}$ in $E_2$ but only in such a way that the vertex of $e$ belonging to $W_{1,2}$ remains of degree one on the path.
Then, we  try to extend $P$ to a copy of $P_3^{(k)}$ in $E_2$, etc., always making sure that the vertices in $W_{1,2}$ are of degree one. Eventually, either we find a blue copy of $P_\ell^{(k)}$ or end up with a blue copy $P$ of $P_t^{(k)}$, for some $1\le t\le\ell-1$, which cannot be further extended. We move the  pendant vertices of $P$ belonging to $W_{1,2}$ to the trash set $T_2$, while the  remaining pendant vertices of $P$ go to $W_{2,2}$. Then we try to extend the  shortened path again.
By moving the pendant vertices of $P$ in $W_{1,2}$ to $T_2$ we make sure that in the next iterations there will be no blue edge~$e$ with exactly one vertex in $W_{1,2}$, one vertex in  $W_{2,2}$ and $(k-2)$ vertices in $V_2\setminus W_{2,2}$.
We terminate Round 2 as soon as
\[
 |W_{2,2}| \ge(k-2)\tau_{2}.
\]
If at some point $P=\emptyset$ and cannot be extended and $ |W_{2,2}| < (k-2)\tau_{2}$, then
we move $\lceil(k-2) \tau_{2}\rceil -|W_{2,2}|$ arbitrary vertices from $V_2$ to $W_{2,2}$ and stop. Note that at the end of this round
\begin{equation}\label{step2:W22:ub}
|W_{2,2}| \le (k-2) \tau_{2} + 2(k-2).
\end{equation}
We  set $S_2:=V(P)$ and $V_2:=V\setminus (W_{1,2}\cup W_{2,2} \cup S_1 \cup S_2 \cup T_2)$. Observe that no edge of \linebreak $K^{(k)}(W_{1,2},W_{2,2},V_2)$ is red or blue. We will now show that
\begin{equation}\label{TW}
|T_2| \le t_2\qquad\mbox{ and }\qquad|W_{1,2}| \ge (k-2) \tau_{2}.
\end{equation}

First observe that the size of $W_{1,1}$ (the set obtained in Round 1) satisfies
\begin{equation}\label{step2:1}
|W_{1,1}| \le |W_{1,2}| + |T_{2}| + \ell - 1.
\end{equation}
Indeed, at the end of this round $W_{1,1}$ is the union of $W_{1,2} \cup T_2$ and the vertices in $V(P)\cap W_{1,2}$ that were moved to $S_2$. Since $|V(P)\cap W_{1,2}| \le \ell-1$, \eqref{step2:1} holds.

Also note that each vertex in $T_2$ can be matched with a set of $k-2$ or $k-1$ vertices in $W_{2,2}$, and all these sets are disjoint. Consequently,
\begin{equation}\label{step2:2}
 |W_{2,2}|\ge (k-2) |T_2|.
\end{equation}
Inequality (\ref{step2:2}) immediately implies that
\[
|T_2| \stackrel{\eqref{step2:2}}{\le} \frac{1}{k-2} |W_{2,2}| \stackrel{\eqref{step2:W22:ub}}{\le}  \tau_2 + 2 = t_2.
\]
Furthermore,
\[
(k-1) \tau_2 +\ell +1 \stackrel{\eqref{step1:W11:lb}}{\le}  |W_{1,1}| \stackrel{\eqref{step2:1}}{\le}  |W_{1,2}| + |T_{2}| + \ell - 1 \le |W_{1,2}| + \tau_2 +\ell+1,
\]
completing the proof of (\ref{TW}).

From now on we proceed inductively. Assume that $i\ge 3$ and we have just finished round $i-1$ constructing so far, for each $1\le j\le i-1$, sets $S_j$, $T_j$, and $W_{j,i-1}$, satisfying
\begin{equation}\label{stepi1:Wji1}
 |W_{j,i-1}|\ge \frac{k-i+1}{i-2} \tau_{i-1},
\end{equation}
$|S_{i-1}| \le |V(P_{\ell-1}^{(k)})|$, and $|T_{i-1}| \le t_{i-1}$, and the residual set
$$V_{i-1} = V\setminus \bigcup_{j=1}^{i-1} (W_{j,i-1} \cup S_j \cup T_j )$$ such that  $K^{(k)}(W_{1,i-1},\dots, W_{i-1,i-1},V_{i-1})$ contains no edge of color $1,2,\dots,$ or $i-1$.

\bigskip

\noindent
\textbf{Round $i$, $3\le i\le r-1$}

\smallskip

\noindent
We begin the $i$th round by resetting $W_{1,i}:=W_{1,i-1},\dots,W_{i-1,i}:=W_{i-1,i-1}$, and $V_i:=V_{i-1}$, and setting $P:=\emptyset$, $W_{i,i}:=\emptyset$, and $T_i:=\emptyset$. We consider only edges of color $i$ in $K^{(k)}(W_{1,i},\dots, W_{i-1,i},V_{i})$. Let us denote the set of such edges by $E_i$.

\begin{figure}
\centering
\includegraphics[scale = 0.8]{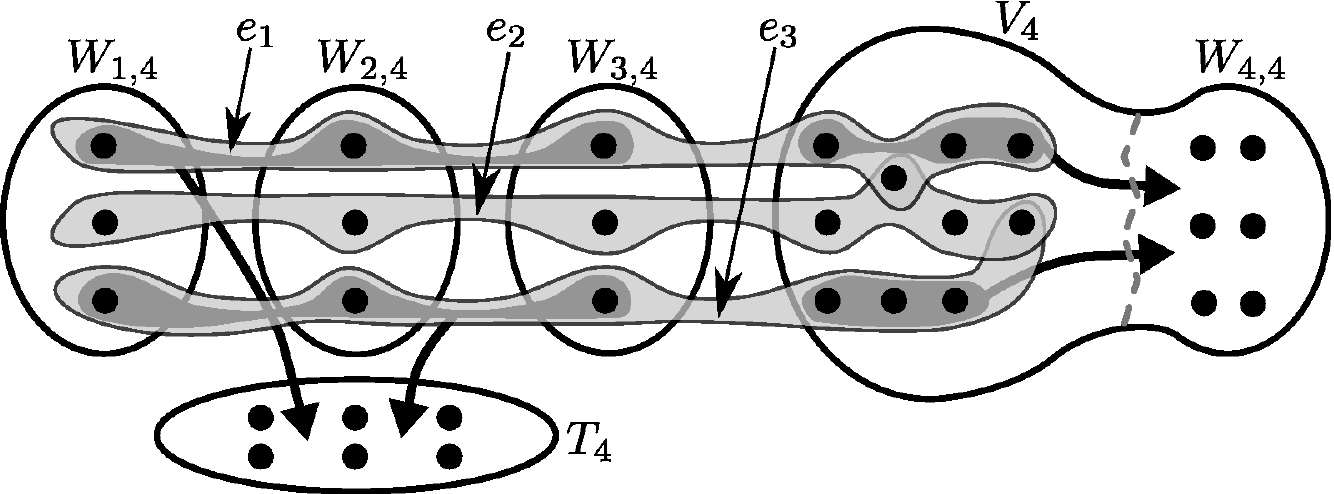}
\caption{Applying the algorithm to  a 7-uniform hypergraph. Here $i=4$ and path $P$, which consists of edges $e_1$, $e_2$, and $e_3$, cannot be extended. Therefore, the vertices in $V(P) \cap (W_{1,4}\cup W_{2,4} \cup W_{3,4})$ are moved to the trash bin $T_4$ and the pendant vertices in $V_4 \cap (e_1\cup e_3)$  are moved to~$W_{4,4}$.}
\label{fig:alg}
\end{figure}

As in the previous steps we are trying to extend the current path $P$ using the edges of $E_i$, but only in such a way that the vertices from $P$ that are in $W_{1,i}\cup\dots\cup W_{i-1,i}$ have degree one in $P$ and the vertices of degree two in $P$ belong to $V_i$. When an extension is no longer possible and $P\neq\emptyset$, we move the  pendant vertices of $P$ belonging to $\bigcup_{j=1}^{i-1}W_{j,i}$ to the trash set $T_i$, while the  remaining pendant vertices of $P$ go to $W_{i,i}$ (see Figure~\ref{fig:alg}). Then we try to extend the shortened path. We terminate the $i$th round as soon as 
\[
 |W_{i,i}|\ge\frac{k-i}{i-1} \tau_{i} .
\]
If $P=\emptyset$ and cannot be extended and $ |W_{i,i}|<\frac{k-i}{i-1} \tau_{i}$, then
we move $\lceil \frac{k-i}{i-1} \tau_{i}\rceil-|W_{i,i}|$  vertices from $V_i$ to $W_{i,i}$ and stop. This yields that
\begin{equation}\label{stepi:Wii:ub}
 |W_{i,i}| \le \frac{k-i}{i-1} \tau_{i}+2(k-i).
\end{equation}
Similarly as in~\eqref{step2:1} and~\eqref{step2:2} notice that for all $1\le j\le i-1$
\begin{equation}\label{stepi:1}
|W_{j,i-1}| \le |W_{j,i}| + \frac{|T_{i}|}{i-1} + \ell - 1
\end{equation}
and
\begin{equation}\label{stepi:2}
|T_i| \le \frac{i-1}{k-i} |W_{i,i}| \le \tau_i + 2(i-1) = t_i.
\end{equation}
Thus,
\[
\frac{k-i+1}{i-2} \tau_{i-1} \stackrel{\eqref{stepi1:Wji1}}{\le}
|W_{j,i-1}|
\stackrel{\eqref{stepi:1},\eqref{stepi:2}}{\le}
|W_{j,i}| + \frac{\tau_i}{i-1} + 2  + \ell-1
= |W_{j,i}| + \frac{\tau_i}{i-1} + \ell+1
\]
and, since also
\[
\frac{k-i+1}{i-2} \tau_{i-1} \stackrel{\eqref{tau}}{=} \frac{k-i+1}{i-1} \tau_i + \ell+1,
\]
we get
\begin{equation}\label{Wji}
 |W_{j,i}|\ge \frac{k-i}{i-1} \tau_i.
\end{equation}
Finally we set $S_i:=V(P)$. Consequently, when the $i$th round ends, we have (\ref{Wji})
for all $1\le j\le i$. We also have $|S_{i}| \le |V(P_{\ell-1}^{(k)})|$, $|T_{i}| \le t_{i}$, and $V_{i} = V \setminus \bigcup_{j=1}^{i} (W_{j,i} \cup S_j \cup T_j )$  such that \linebreak $K^{(k)}(W_{1,i},\dots, W_{i-1,i},W_{i,i},V_{i})$ has no edges of color $1,2,\dots,$ or $i$.

In particular, when the $(r-1)$st round is finished, we have, for each $1\le j\le r-1$,
\begin{equation}\label{stepr1:Wjr1}
|W_{j,r-1}| \ge \frac{k-r+1}{r-2} \tau_{r-1},
\end{equation}
$|S_{r-1}| \le |V(P_{\ell-1}^{(k)})|$ and $|T_{r-1}| \le t_{r-1}$. Set $W_j := W_{j,r-1}$, $j=1,\dots,r-1$,  and $V_r := V\setminus \bigcup_{j=1}^{r-1} (W_{j} \cup S_j \cup T_j )$ and observe that $K^{(k)}(W_{1},\dots, W_{r-1},V_{r})$
has only edges of color~$r$.

By (\ref{stepr1:Wjr1}), for each $1\le j\le r-1$
\[
|W_{j}| \stackrel{\eqref{stepr1:Wjr1}}{\ge} \frac{k-r+1}{r-2} \tau_{r-1} \stackrel{\eqref{tau}}{=} \ell.
\]

Now we are going to show that $|V_r| \ge \ell(k-r+1)$ which will complete the proof as this bound yields a monochromatic copy of $P_\ell^{(k)}$ inside $K^{(k)}(W_{1},\dots, W_{r-1},V_{r})$. (Actually for $r\le k-1$ it suffices to show that $|V_r| \ge \ell(k-r)+1$.)

First observe that
\begin{equation}\label{eq:WandT}
|W_{1,1}| + \dots + |W_{r-2,r-2}| \ge |W_{1}|+\dots+|W_{r-2}| + |T_1| + \dots + |T_{r-1}|.
\end{equation}
This is easy to see, since during the process
\[W_{i,i} \supseteq W_{i,r-1} \cup \left(W_{i,i}\cap (T_{i+1}\cup\dots\cup T_{r-1})\right).
\]
Also,
\begin{align*}
|W_{1,1}| &\stackrel{\eqref{step1:W11:ub}}{\le} (k-1)\tau_2+ 2(k-1) + \ell+1\\
&\stackrel{(\ref{claim5})}{\le} (k-1)(\ell+1)\left(\frac{1}{k-r+1}+\ln\left(1+\frac{r-2}{k-r+1}\right)\right) + 2(k-1) + \ell+1
\end{align*}
and, for $2\le i\le r-1$,
\begin{align*}
|W_{i,i}| &\stackrel{\eqref{stepi:Wii:ub}}{\le} \frac{k-i}{i-1} \tau_{i}+2(k-i)\\
&\stackrel{(\ref{claim5})}{\le} (k-i)(\ell+1)\left(\frac{1}{k-r+1}+\ln\left(1+\frac{r-2}{k-r+1}\right)\right) + 2(k-i).
\end{align*}
Since
\[
\sum_{i=1}^{r-1} (k-i) = (k-r/2)(r-1),
\]
we have by~\eqref{eq:WandT} that
\begin{align*}
|W_{1}|+\dots&+|W_{r-1}| + |T_2| + \dots + |T_{r-1}|\\
&\le (\ell+1)\left(\frac{1}{k-r+1}+\ln\left(1+\frac{r-2}{k-r+1}\right)\right)(k-r/2)(r-1)\\
&\hspace{6cm}+ (2k-r)(r-1) + \ell+1\\
&\le k(\ell+1)r\left(\frac{1}{k-r+1}+\ln\left(1+\frac{r-2}{k-r+1}\right)\right)\\
&\hspace{6cm}+ (2k-r)(r-1) + \ell+1.
\end{align*}
As also $|S_i| \le |V(P_{\ell-1}^{(k)})| =(k-1)(\ell-1)+1$ for each $1\le i \le r-1$
and
$$|V_r|=|V|-\sum_{i=1}^{r-1}(|W_i|+|T_i|+|S_i|),$$
 we finally obtain, using  the lower bound on $n=|V|$, that

\begin{align*}
|V_r|
&\ge  k(\ell+1)r - (2k-r)(r-1) - \ell - 1-  (r-1)\left[(k-1)(\ell-1)+1\right] \\
&= \ell(2r-3) + (r-1)(r-2) + (k-1) + \ell(k-r+1) \ge \ell(k-r+1),
\end{align*}
since the first three terms in the last line are nonnegative.

\bigskip

To check the $O(n^k)$ complexity time, observe that in the worst-case scenario we need to go over all edges colored by the first $r-1$ colors and no edge is visited more than once.

\section{Auxiliary inequalities}

For the sake of completeness we prove here two straightforward inequalities.
\begin{claim}\label{claim:tau}
Let $2\le r\le k$, $1\le i\le r-1$ and
\begin{equation*}
\tau_i =
\begin{cases}
(i-1)\left(  \frac{\ell}{k-r+1} + \frac{\ell+1}{k-r+2} + \dots + \frac{\ell+1}{k-i} \right) & \text{ if } 1\le i \le r-2,\\[4pt]
(r-2)\frac{\ell}{k-r+1} & \text{ if } i=r-1.
\end{cases}
\end{equation*}
Then,
\[
\tau_i \le (i-1)(\ell+1) \left(\frac{1}{k-r+1}+\ln\left(1+\frac{r-2}{k-r+1}\right)\right).
\]
\end{claim}
\begin{proof}
It suffices to observe that
\begin{align*}
\frac{1}{k-r+1} + \frac{1}{k-r+2}  + \dots + \frac{1}{k-i}
&\le \frac{1}{k-r+1}+ \int_{k-r+1}^{k-i} \frac{dx}{x} \\
&= \frac{1}{k-r+1} + \ln\left( \frac{k-i}{k-r+1}\right) \\
&\le \frac{1}{k-r+1} +\ln\left( \frac{k-1}{k-r+1}\right)\\
&= \frac{1}{k-r+1} + \ln\left( 1 + \frac{r-2}{k-r+1}\right).
\end{align*}
\end{proof}

\begin{claim}\label{claim:cor}
For all $2\le r \le k-1$ we have
\begin{equation}\label{claim:ineq}
\frac{1}{k-r+1}+\ln\left(1+\frac{r-2}{k-r+1}\right) \le \ln\left(1+\frac{r-1}{k-r}\right).
\end{equation}
\end{claim}
\begin{proof}
Let $f(x) = \ln\left( 1+\frac{1}{x}\right) - \frac{1}{x+1}$
and observe that $f'(x) = \frac{-1}{x(x+1)^2}$. Hence, $f(x)$ is decreasing for $x>0$ and so $f(x) \ge \lim_{x\to\infty} f(x) = 0$.
Consequently, for $x = k-r$ (by assumption $k-r\ge 1$) we get that
\[
\frac{1}{k-r+1} \le  \ln\left( 1+\frac{1}{k-r}\right)
= \ln\left(\frac{k-r+1}{k-r}\right)
= \ln\left(\frac{k-1}{k-r}\right) - \ln\left(\frac{k-1}{k-r+1}\right),
\]
which is equivalent to~\eqref{claim:ineq}.
\end{proof}


\end{document}